\documentclass[11pt]{article}
\usepackage{amsfonts,amssymb,amsmath,amscd,latexsym,makeidx,amsthm, color,hyperref,graphicx,multicol,pdfsync}
\usepackage[utf8]{inputenc}
\usepackage[english]{babel}
\newtheorem{teo}{Theorem}[section]
\newtheorem{lemma}{Lemma}[section]
\newtheorem{cor}{Corollary}[section]
\newtheorem{oss}{Remark}[section]
\newtheorem{prop}{Proposition}[section]

\newcommand{\ra}{\longrightarrow}
\newcommand{\dis}{\displaystyle}
\newcommand{\bs}{\backslash}

\newcommand{\ov}{\overline}

\newcommand{\R}{\mathbb{R}}
\newcommand{\ti}{\widetilde}

\newcommand{\N}{\mathbb{N}}

\newcommand{\eps}{\varepsilon}

\newcommand{\ph}{\varphi}

\DeclareMathOperator{\sgn}{sgn}
\newcommand{\rw}{\rightharpoonup}

\newenvironment{Si}[1]{\left\{\begin{array}{#1}}{\end{array} \right. }

\makeindex 
 
\usepackage{authblk}
\DeclareMathOperator{\ind}{ind}
\newtheorem{Te}{Theorem}

\title{Singular Liouville Equations on $S^2$: Sharp Inequalities and Existence Results}
\author{Gabriele Mancini\thanks{S.I.S.S.A/I.S.A.S, Via Bonomea 265, 34136 Trieste (Italy) - gmancini@sissa.it\\The author is supported by the PRIN project {\em Variational and perturbative aspects of nonlinear differential problems}.}}
\date{}
\begin{document}
\maketitle

\begin{abstract}We prove a sharp Onofri-type inequality and non-existence of extremals for a Moser-Tudinger functional on $S^2$ in the presence of potential having positive order singularities. We also investigate the existence of critical points and give some sufficient conditions under symmetry or nondegeneracy assumptions. 
\end{abstract}

\section{Introduction}
In this work we study sharp Onofri-type inequalities on the standard Euclidean sphere $(S^2,g_0)$, and existence of critical points for a singular Moser-Trudinger functional. Given a smooth, closed surface $\Sigma$, and $m$ points $p_1,\ldots,p_m\in \Sigma$,   we consider the functional  
\begin{equation}\label{funz mos}
J_\rho^h(u)= \frac{1}{2}\int_{\Sigma}|\nabla u|^2 dv_{g} + \frac{\rho}{|\Sigma|} \int_{\Sigma} u\; dv_{g}-\rho\log\left(\frac{1}{|\Sigma|}\int_{\Sigma} h e^{u} dv_{g}\right)
\end{equation}
where $h$ is  a positive singular  potential satisfying 
\begin{equation}\label{h1}
h\in C^\infty (\Sigma\bs \{p_1,\ldots,p_m\}) \qquad \mbox{ and } \quad h(x)\approx d(x,p_i)^{2\alpha}\mbox{ with } \alpha_i>-1 \mbox{ near }p_i,
\end{equation} $i=1,\ldots,m$. Functionals of this this kind were first introduced, for the regular case $m=0$, by Moser (\cite{mos}, \cite{mos2}), in connection to the study of the Gaussian curvature equation on compact surfaces and Nirenberg's problem on $S^2$. They also have a role in spectral analysis due to Polyakov's formula (see \cite{Poly1}, \cite{Poly2}, \cite{OPS}, \cite{OPS2}). In the case $m>0$, the functional \eqref{funz mos} appears in the problem of prescribing the Gaussian curvature  of Riemannian metrics with conical singularities.  We recall that a metric on $\Sigma$ with conical singularities of order $\alpha_1,\ldots, \alpha_m>-1$ in $ p_1,\ldots,p_m$, is a metric of the form $e^{u} g$ where $g$ is smooth metric on $\Sigma$, and $u\in C^\infty(\Sigma\bs\{p_1,\ldots,p_m\})$ satisfies 
$$
|u(x)+2\alpha_i \log d(x,p_i)|\le C \qquad \mbox{ near } p_i, \; i=1,\ldots,m.
$$ 
It is possible to prove (see Proposition 2.1 in \cite{BarMalDem}) that a metric of this form has Gaussian curvature $K$ if and only if $u$ is a distributional solution of the Gaussian curvature equation 
\begin{equation}\label{Gauss}
-\Delta_{g} u = 2 K e^{u} - 2K_{g} -4\pi \sum_{i=1}^m\alpha_i \delta_{p_i}.
\end{equation}
where $K_{g}$ is the Gaussian curvature of $(\Sigma,g)$.  If $\chi(\Sigma)+ \sum_{i=1}^m \alpha_i\neq 0$ and $K_g$ is constant, \eqref{Gauss} is equivalent to the singular Liouville equation
\begin{equation}\label{Lio sing}
-\Delta_g u= \rho\left(\frac{K e^u}{\int_{\Sigma} K e^u dv_g }-\frac{1}{|\Sigma|}  \right) -4\pi \sum_{i=1}^m \alpha_i \left(\delta_{p_l}-\frac{1}{|\Sigma|}\right)
\end{equation} 
for 
\begin{equation}\label{rho geom}
\rho = \rho_{geom}:= 4\pi \left( \chi(\Sigma)+ \sum_{i=1}^m\alpha_i \right).
\end{equation}
Denoting by $G$ the Green's function of $(\Sigma,g)$, that is the solution of
\begin{equation}\label{Gr}
\begin{Si}{c}
-\Delta_g G(x,\cdot) =\delta_x \; \mbox{ on }\; \Sigma\\
\int_\Sigma G(x,y)dv_{g}(y)=0,
\end{Si}
\end{equation}
the change of variable $u\longleftrightarrow u+4\pi \sum_{i=1}^m\alpha_i G(x,p_i)$ reduces \eqref{Lio sing} to 
\begin{equation}\label{regular}
-\Delta_{g} u  = \rho \left( \frac{ h e^{u}}{\int_{\Sigma} h e^{u}dv_g} -\frac{1}{|\Sigma|}\right)
\end{equation}
that is the Euler-Lagrange equation of the functional \eqref{funz mos} corresponding to the potential
\begin{equation}\label{h e K}
h(x)= K e^{-4\pi \sum_{i=1}^m \alpha_i G_{p_i}},
\end{equation}
which satisfies \eqref{h1}. Equations \eqref{Lio sing} and \eqref{regular} have also been widely studied in  mathematical physics. For example, they appear in the description of Abelian vortices in Chern-Simmons-Higgs theory, and have applications in Superconductivity and Electroweak theory (\cite{Tarantello}, \cite{HongKim}). We refer to \cite{BarMal}, \cite{CarMal1}, \cite{CarMal2}, \cite{MalRui}, \cite{Carl}, \cite{ChenLin3}, \cite{ChenLin4}, for some recent existence results.

A fundamental role in the variational  analysis of \eqref{funz mos} is played by singular versions of the standard Moser-Trudinger inequality (see \cite{mos,troyanov}). In \cite{troyanov}, Troyanov  proved  that, for  every function $h$ satisfying \eqref{h1}, there exists a constant $C= C(h,g,\Sigma)$ such that
\begin{equation}\label{troy}
\log \left(\frac{1}{|\Sigma|}\int_{\Sigma} h e^{u-\ov{u}} dv_g  \right)\le \frac{1}{16\pi(1+\ov{\alpha})}\int_{\Sigma}|\nabla u|^2 dv_{g} + C(h,\Sigma,g)
\end{equation}
 $\forall\; u\in H^1(\Sigma)$, where $\dis{\ov{\alpha}=\min\left\{0,\min_{1\le i\le m}\alpha_i\right\}}$.
In particular the functional $J_\rho^h$ is bounded from below $\forall\; \rho \in (0,8\pi(1+\ov{\alpha})]$ and is coercive for $\rho \in (0,8\pi(1+\ov{\alpha}))$. Furthermore, it is possible to prove that the constant $\frac{1}{16\pi}$ is sharp, that is 
$$
\inf_{H^1(\Sigma)} J_\rho^h = -\infty \quad \forall\; \rho >8\pi (1+\ov{\alpha}).
$$
In the special case $m=0$, $h\equiv 1$ and $(\Sigma,g)=(S^2,g_0)$, a sharp version of \eqref{troy} was proved by Onofri in \cite{onofri}:
\begin{Te}[{\bf Onofri's inequality \cite{onofri}}]\label{onofri} $\forall\; u\in H^1(S^2)$ we have
$$
\log \left(\frac{1}{4\pi}\int_{S^2} e^{u-\ov{u}} dv_{g_0}  \right)\le \frac{1}{16\pi}\int_{S^2}|\nabla u|^2 dv_{g_0}
$$
with equality holding if and only if $e^{u}g_0$ is a metric on $S^2$ with positive constant Gaussian curvature, or, equivalently, $u=\log |\det d\ph| +c$ with $c\in \R$ and  $\ph:S^2 \ra  S^2$ a conformal diffeomorphism of $S^2$.
\end{Te}

Motivated by this result, in \cite{mio2} we started the study Onofri-type inequalities and existence of energy-minimizing solutions on $S^2$ for the potential 
$$
h(x)= e^{-4\pi \sum_{i=1}^m\alpha_i G(x,p_i)}
$$
(i.e. \eqref{h e K} with $K\equiv 1$), and we extended Theorem \ref{onofri}  to the cases $m=1$, and $m=2$ with $\min\{\alpha_1,\alpha_2\}<0$. 

\begin{Te}[\cite{mio2}]\label{una sing2}
If $h=e^{-4\pi \alpha G_p}$ with $\alpha\neq 0$, then $\forall\; u\in H^1(\Sigma)$
$$\log\left(\frac{1}{4\pi}\int_{S^2} h e^{u-\ov{u}}dv_{g_0}\right)<\frac{1}{16\pi\min\{1,1+\alpha\}}\int_{S^2} |\nabla u|^2dv_{g_0}+\max\left\{\alpha,-\log(1+\alpha)\right\}.$$ Moreover equation \eqref{regular} has no solution for $\rho= 8\pi\min\{1,1+\alpha\}$. 
\end{Te}
\vspace{0.1cm}
\begin{Te}[\cite{mio2}] \label{due sing2}
If $h=e^{-4\pi \alpha_1 G_p-4\pi \alpha_2 G_{p_2}}$ with $p_2= -p_1$, $\alpha_1= \min\{\alpha_1,\alpha_2\}<0$, then $\forall\; u\in H^1(\Sigma)$
$$\log\left(\frac{1}{4\pi}\int_{S^2} h e^{u-\ov{u}}dv_{g_0}\right)\le \frac{1}{16\pi(1+\alpha_1)}\int_{S^2} |\nabla u|^2dv_{g_0}+ \alpha_2-\log(1+\alpha_1).$$
If $\alpha_1\neq \alpha_2$ there is no function realizing equality and no solution of \eqref{regular} for $\rho=8\pi(1+\alpha_1)$, while if $\alpha_1=\alpha_2$ then equality holds for $u$ if and only if the following equivalent conditions are satisfied: 
\begin{itemize}
\item $u$ is a solution of \eqref{regular} for $\rho =8\pi(1+\alpha_1)$.
\item  $h e^{u} g$ is a metric with constant positive Gaussian curvature and conical singularities of order $\alpha_i$ in $p_i$, $i=1,2$.
\item If $\pi$ denotes the stereographic projection from $p_1$, then $$
u\circ \pi^{-1}(y)= 2\log\left( \frac{(1+|y|^2)^{1+\alpha_1}}{1+e^\lambda|y|^{2(1+\alpha_1)}} \right)+c
$$
for some $\lambda,\; c \in \R$. 
\end{itemize}
\end{Te}

We stress that the critical parameter $\rho = 8\pi(1+\ov{\alpha})$ is generally different from the geometric parameter \eqref{rho geom} (except for some special cases, for example $m=2$ and $\alpha_1=\alpha_2<0$), thus critical points cannot always be interpreted in terms of metrics with prescribed curvature. 

In this paper we will assume \eqref{h e K} with  $\alpha_i\ge 0$ for $1\le i\le m$ and 
$$
K\in C^\infty_+(S^2):=\left\{f\in C^\infty(S^2)\;:\quad f(x)>0\quad \forall\; x\in S^2\right\}.
$$ Our first result  is a further extension of Onofri's inequality. 

\begin{teo}\label{sing pos}
Assume that $h$ satisfies \eqref{h e K} with $K\in C^\infty_+(S^2)$ and $\alpha_1,\ldots,\alpha_m\ge 0$, then
$$
\inf_{H^1(S^2)} J_{8\pi}^h = -8\pi \log \max_{S^2} h.
$$
Moreover $J_\rho^h$ has no minimum point, unless $\alpha_1=\ldots=\alpha_m=0$ (or, equivalently, $m=0$) and   $K$ is constant. 
\end{teo}

Clearly,  the sharp value of the constant $C(h,S^2,g_0)$  is given by 
$$
C(h,S^2,g_0)=-\frac{1}{8\pi(1+\ov{\alpha})}\inf_{H^1(S^2)} J_{8\pi(1+\ov{\alpha})}^h,
$$
thus  Theorem \ref{sing pos} is equivalent to the following sharp inequality:

\begin{cor}
If $h$ satisfies \eqref{h e K} with $K\in C^\infty_+(S^2)$ and $\alpha_1,\ldots,\alpha_m\ge 0$, then  $\forall\; u\in H^1(S^2)$ we have 
$$\log\left(\frac{1}{4\pi}\int_{S^2} h e^{u-\ov{u}}dv_{g_0}\right)\le \frac{1}{16\pi}\int_{S^2} |\nabla u|^2dv_{g_0}+ \log\max_{S^2} h$$
with equality holding if and only if $m=0$, $K$ is constant and  $u$ realizes equality in Theorem \ref{onofri}.
\end{cor}

Theorem \ref{sing pos} states that $J_{8\pi}^h$  has no minimum point, but does not exclude the existence of different kinds on critical points. In contrast to Theorem \ref{due sing2}, if $\alpha_i>0$ for $1\le i \le m$, we will show that in many cases it is possible to find saddle points of $J_{8\pi}^h$. A simple example is given by the case in which $h$ is axially symmetric. In this case an improved Moser-Trudinger inequality allows to minimize $J_{8\pi}^h$ in the class of axially symmetric functions and find a solution of \eqref{regular}.

\begin{teo}\label{teo 2}
Assume that $h$ satisfies \eqref{h e K} with $m=2$, $p_1=-p_2$, $\min\{\alpha_1,\alpha_2\}=\alpha_1>0$ and $K\in C^\infty_+(S^2)$ axially symmetric with respect to the direction identified by $p_1$ and $p_2$. Then the Liouville equation \eqref{regular} has an axially symmetric solution $\forall \rho \in (0,8\pi(1+\alpha_1))$.
\end{teo}

In the last part of the paper  we prove further general existence results using the Leray-Schauder degree theory introduced  in \cite{Li}, \cite{ChenLin1},  \cite{ChenLin2},  \cite{ChenLin3} and  \cite{ChenLin4}. Solutions of \eqref{regular} on the space 
\begin{equation}\label{spazioH0}
H_0:=\left\{u\in H^1(S^2)\;: \;\int_{S^2} u\; dv_{g_0} =0\right\}.
\end{equation}
can be obtained  as solutions of $T_\rho(u)+u=0$ where $T_\rho:H_0\ra H_0$ is defined by
\begin{equation}\label{Trho}
T_\rho(u)= \Delta_{g_0}^{-1}\left( \frac{h e^{u}}{\int_{S^2} h e^{u} dv_{g_0}}- \frac{1}{4\pi}\right).
\end{equation}
In \cite{ChenLin4}, Chen and Lin computed  the Leray-Schauder degree
\begin{equation}\label{drho}
d_\rho= deg_{LS}(Id+ T_\rho,0, B_R(0))
\end{equation}
for
\begin{equation}\label{sing set}
\rho\notin \Gamma(\alpha_1,\ldots,\alpha_n):=\left\{ 8\pi k_0 + 8\pi \sum_{i=1}^m k_i (1+\alpha_i)\;:\; k_0\in \mathbb{N}, k_i \in \{0,1\},\sum_{i=0}^m k_i>0\right\}.\end{equation}
If $m\ge 2$, one has $d_\rho\neq 0$ for any $\rho \in (0,8\pi (1+\ov{\alpha}))\bs 8\pi \N$. While Theorem \ref{sing pos} implies blow-up  as $\rho \nearrow 8\pi$, we can find solutions for $\rho=8\pi$ by taking $\rho \searrow 8\pi$, provided the  Laplacian of $K$ is not too large at the critical points of $h$.

\begin{teo}\label{teo 3}
If  $h$ satisfies \eqref{h e K} with $K\in C^\infty_+(S^2)$, $m\ge 2$,  $\alpha_1,\ldots,\alpha_m>0$ and
\begin{equation}\label{cond 1}
\Delta_{g_0} \log K (x)< \sum_{i=1}^m\alpha_i 
\end{equation}  
$\forall\; x\in \Sigma$ such that $\nabla h (x)= 0$,
then equation \eqref{regular} has a solution for $\rho = 8 \pi$. 
\end{teo}

The same strategy can be used for $\rho= 8 k\pi $, with $k<1+\alpha_1$.

\begin{teo}\label{teo 31}
If  $h$ satisfies \eqref{h e K} with $K\in C^\infty_+(S^2)$, $m\ge 2$,  $\;0<\alpha_1\le \ldots\le\alpha_m$ and
\begin{equation}\label{cond 2} 
\Delta_{g_0} \log K (x)< \sum_{i=1}^m\alpha_i +2(1-k)
\end{equation}
$\forall\; x\in S^2$, then equation \eqref{regular} has a solution for $\rho = 8 k \pi$, $k<1+\alpha_1$. 
\end{teo}

Note that Theorems \ref{teo 3} and \ref{teo 31} can be applied in the case $K\equiv 1$. If the sign condition \eqref{cond 1} is not satisfied, then it is not possible to exclude blow-up of solutions as $\rho \ra 8\pi$.  However, as it is pointed out in the introduction of \cite{ChenLin2}, under some non-degeneracy assumptions on $h$, the Leray Schauder degree $d_{8\pi}$  is well defined and can be explicitly computed by  taking into account the contributions of all the blowing-up families of solutions.  In particular one can prove that $d_{8\pi}\neq 0$ under one of the following conditions. 

\begin{teo}\label{teo 4}
Let $h$ be a Morse function on $S^2 \bs\{p_1,\ldots,p_m\}$ satisfying \eqref{h e K} with $K\in C^\infty_+(S^2)$, $m\ge 0$, $\alpha_1,\ldots,\alpha_m>0$ and assume $\Delta_{g_0} \log h \neq 0$ at all the critical points of $h$. If $h$  has $r$ local maxima and $s$ saddle points in which $\Delta_{g_0} h <0$, then equation \eqref{regular} has a solution for $\rho =8\pi$ provided $r \neq s+1$.
\end{teo}

\begin{teo}\label{teo 5}
Let $h$ be a Morse function on $S^2 \bs\{p_1,\ldots,p_m\}$ satisfying \eqref{h e K} with $K\in C^\infty_+(S^2)$, $m\ge 0$, $\alpha_1,\ldots,\alpha_m>0$ and assume $\Delta_{g_0} \log h \neq 0$ at all the critical points of $h$. If $h$  has $r'$ local minima in $S^2\bs\{p_1,\ldots,p_m\}$ and $s'$ saddle points in which $\Delta_{g_0} h >0$, then equation \eqref{regular} has a solution for $\rho =8\pi$ provided $s'\neq r'+\ov{d}$, where 
$$
\ov{d}:=d_{8\pi +\eps}=\begin{Si}{cc}
2 & \; m\ge 2,\\
0 & \; m=1,\\
-1 & \; m=0.
\end{Si}
$$
\end{teo}

In the regular case $m=0$, Theorem \ref{teo 4} was first proved by Chang and Yang in \cite{CY} using a min-max scheme. A different proof was later given by Struwe \cite{StrCY} through a geometric flow approach. 

\section{Proof of the Main Results}
The proof of Theorem \ref{sing pos} is a rather simple  consequence of Theorem \ref{onofri}.

\begin{proof}[Proof of Theorem \ref{sing pos}]
Let us consider $$J^1_{8\pi}(u):=\frac{1}{2}\int_{S^2}|\nabla u|^2 dv_{g_0} + 2\int_{S^2} u \; dv_{g_0}-8\pi \log\left(\frac{1}{4\pi}\int_{S^2}  e^{u} dv_{g_0}\right).$$ By Theorem \ref{onofri} we have $J^1_{8\pi}(u)\ge0$ $\forall\; u\in H^1(S^2)$. 
The condition $\alpha_1,\ldots,\alpha_m\ge 0$ guarantees $h\in C^0(S^2)$. Thus we have
\begin{equation}\label{dis}
J_{8\pi}^h(u) \ge \frac{1}{2}\int_{S^2}|\nabla u|^2 dv_{g_0} + 2\int_{S^2} u \;dv_{g_0}-8\pi \log\left(\frac{1}{4\pi}\max_{S^2} h \int_{S^2}  e^{u} dv_{g_0}\right)=
\end{equation}
$$
= J^1_{8\pi}(u)-8\pi \log \max_{S^2} h  \ge -8\pi \log \max_{S^2} h.
$$
Since  $e^{u}>0$ on $S^2$, equality  can hold only if
$$
h \equiv \max_{S^2} h
$$
which, by \eqref{h e K}, is possible only if $\alpha_1=\ldots=\alpha_m =0$ and $K$ is constant.  
To complete the proof it is sufficient to observe that the lower bound in \eqref{dis} is sharp. Let us fix a point $p\in S^2$ such that $\dis{h(p)=\max_{S^2} h}$, and   consider the stereographic projection $\pi:S^{2}\bs\{p\}\ra \R^2$. For $t>0$ we define $u_t:=\log |\det d\ph_t|$, where $\{\ph_t\}_{t>0}$ is the family of conformal  diffeomorphisms of $S^2$ that, in the local coordinates determined by $\pi$, corresponds to  the family of dilations of $\R^2$, namely 
$$
\pi(\ph_t(\pi^{-1}(y)))= t y \qquad \forall\; y\in \R^2.
$$  
By Theorem \ref{onofri}, we have $J^1_{8\pi}(u_t)=0$ $\; \forall\; t>0$. Moreover it is straightforward to verify that 
$$
\int_{S^2} e^{u_t} dv_{g_0} = 4\pi
$$
and $e^{u_t} \rw 4\pi \delta_p$ weakly as measures on $S^2$ for $t\to \infty$. Thus, one has 
$$
J_{8\pi}^h(u_t)= -8\pi \log\left(\frac{1}{4\pi}\int_{S^2} h e^{u_t} dv_{g_0} \right) \stackrel{t\to \infty
}{\ra} -8\pi \log h(p)= -8\pi \log \max_{S^2} h.
$$
\end{proof}

Let us now focus on the case of two antipodal singular points $p_1=-p_2$.  Given any point $p\in S^2 \subset \R^3$ we consider the space 
$$
H_{rad,p} :=\left\{ u\in H^{1}(S^2)\;:\; \exists \; \ph:[-1,1]\ra \R \; \mbox{ measurable s.t. }\; u(x)= v(x\cdot p)\mbox{ for a.e. }x\in S^2\right\}.
$$

\begin{lemma}\label{sym}
Suppose $m=2$, $\min\{\alpha_1,\alpha_2\}=\alpha_1>0$ and $p_2=-p_1$. If $h$ is a positive function satisfying \eqref{h1}, then the Moser-Trudinger functional $J_{\rho}^h$ is bounded from below on $H_{rad,p_1}$ for any $\rho \in (0,8\pi (1+\alpha_1))$.
\end{lemma}
\begin{proof}
Let us consider 
$$
\ti{h}(x):= e^{-4\pi \alpha_1(G(x,p_1)+G(x,p_2))}.
$$
Since $\dis{h= K e^{-4\pi\alpha_1 G(x, p_1)-4\pi \alpha_2 G(x,p_2)} \le \ti{h} \max_{x\in S^2} K(x) e^{4\pi (\alpha_1-\alpha_2) G(x,p_2)} }$ it is sufficient to prove that the functional 
$$
\ti{J}_\rho(u):=J_{\rho}^{\ti{h}}(u)=\frac{1}{2}\int_{S^2} |\nabla u|^2 dv_{g_0} + \frac{\rho}{4\pi}\int_{S^2} u \;dv_{g_0} -\rho \log \left( \frac{1}{4\pi} \int_{S^2} \ti{h} e^{u} dv_{g_0}\right)
$$
is bounded from below for any $\rho <8\pi (1+\alpha_1)$. 
Let us consider Euclidean coordinates $(x_1,x_2,x_3)$ on $S^2$ such that $p_1=(0,0,-1)$, $p_2=(0,0,1)$, and 
let $ \pi $ be the stereographic projection from the point $p_2$. Given a function $u\in H^1(S^2)$ we define $v(|y|):= (u (\pi^{-1}(y)))$, $v_{\alpha_1}(y):= v(|y|^{\frac{1}{{1+{\alpha_1}}}})$ and  $u_{\alpha_1}(x):= v_{\alpha_1}(|\pi(x)|)$. Then we have
\begin{equation}\label{grad}
\int_{S^2}|\nabla u|^2 dv_{g_0}= 2\pi \int_{0}^\infty t |v'(t)|^2 dt   = (1+{\alpha_1})\int_{0}^{+\infty} s |v'_{\alpha_1}(s)|^2 ds = (1+{\alpha_1})\int_{S^2}  |\nabla u_{\alpha_1}|^2 dv_{g_0},
\end{equation}
and, using that $\dis{\sup_{t>0}\frac{1+t^{2(1+{\alpha_1})}}{(1+t^2)^{1+{\alpha_1}}}}<+\infty
$,
$$
\int_{S^2} \ti{h} e^{u} {dv_{g_0}} = 8\pi \int_{0}^{+\infty} e^{2{\alpha_1}} \frac{t^{2{\alpha_1}+1} e^{v(t)}}{(1+t^2)^{2(1+{\alpha_1})}} dt  \le c_{\alpha_1}  \int_{0}^{+\infty} \frac{t^{2{\alpha_1}+1} e^{v_{\alpha_1}(t^{1+{\alpha_1}})}}{(1+t^{2(1+{\alpha_1})})^{2}} dt =$$
\begin{equation}\label{exp}
=   4 \ti{c}_{\alpha_1} \int_{0}^{+\infty} \frac{s e^{v_{\alpha_1}(s)}}{(1+s^2)^2}  =  \ti{c}_{\alpha_1} \int_{S^2} e^{v_{\alpha_1}} dv_{g_0}.
\end{equation}
 
Finally, $\forall\; \eps >0$, $t\in \R^+$
$$
|v(t)-v_{\alpha_1}(t)|\le \left| \int_{t}^{t^\frac{1}{1+{\alpha_1}}} \hspace{-0.1cm}|v'(s)|ds\right| \le \left|  \int_{t}^{t^\frac{1}{1+{\alpha_1}}} \hspace{-0.1cm} s |v'(s)|^2 ds \right|^\frac{1}{2} \left|\frac{{\alpha_1}}{1+{\alpha_1}} \log t \right|\le$$
$$
\le  \frac{\eps}{4\pi} \|\nabla u\|_2^2 + c_{\eps ,{\alpha_1}}|\log t|
$$
from which 
\begin{equation}\label{media}
\left|\int_{S^2} u \; dv_{g_0} - \int_{\Sigma} u_{\alpha_1} \;  dv_{g_0} \right| \le 8\pi \int_{0}^{+\infty} \frac{|v(t)-v_{\alpha_1}(t)|}{(1+t^2)^2} \le \eps  \|\nabla u\|_2^2 + C_{\eps ,{\alpha_1}}.
\end{equation}
\eqref{grad}, \eqref{exp}, \eqref{media} and the Moser-Trudinger inequality \eqref{troy} imply 
$$
\ti{J}_\rho(u) \ge  (1+{\alpha_1}) \left(\frac{1}{2}- \rho \;\eps \right)\int_{S^2}|\nabla u_{\alpha_1}|^2 dv_{g_0} +\rho \int_{S^2} u_{\alpha_1} dv_{g_0} - \rho\log \left(\frac{1}{4\pi}\int_{S^2} e^{u_{\alpha_1}} dv_{g_0}\right) - C_{\epsilon,{\alpha_1},\rho} =
$$
$$
=(1+{\alpha_1}) \left( \left(\frac{1}{2}- \rho \;\eps \right)\int_{S^2}|\nabla u_{\alpha_1}|^2 dv_{g_0} -\frac{\rho}{1+{\alpha_1}}\log \left(\frac{1}{4\pi}\int_{S^2}\hspace{-0.14cm} e^{u_{\alpha_1}-\ov{u}_{\alpha_1}} dv_{g_0}\right) \right) - C_{\epsilon,{\alpha_1},\rho} \ge -  \ti{C}_{\epsilon,{\alpha_1},\rho}
$$
if $\rho<8\pi(1+\alpha_1)$ and $\eps$ is sufficiently small.
\end{proof}
 
\begin{oss}
Arguing as in $\cite{mio2}$, it is possible to describe  the behavior of sequences of minimum points of $J_{\rho}^h$ in $H^{1}_{rad,p_1}(S^2)$ as $\rho \nearrow 8\pi(1+\alpha_1)$ to prove that  also $J_{8\pi (1+\alpha_1)}^h$ is bounded from below. Moreover if $K\equiv 1$ and $\alpha_1= \alpha_2 = \alpha$ then we have
$$
\log \left(\frac{1}{4\pi} \int_{S^2} h e^{u-\ov{u}}dv_{g_0}\right)\le \frac{1}{16\pi (1+\alpha)} \int_{S^2} |\nabla u|^2 dv_{g_0} +\alpha -\log(1+\alpha)\quad\forall\;u\in H_{rad,p_1}(S^2),
$$
with equality holding for 
$$
u\circ \pi^{-1}(y)= 2\log\left( \frac{(1+|y|^2)^{1+\alpha}}{1+e^\lambda|y|^{2(1+\alpha)}} \right)+c,
$$
where $\lambda, c\in \R$ and $\pi$ is the stereographic projection from $p_1$.
\end{oss}
 
\begin{proof}[Proof of Theorem \ref{teo 2}]
By Lemma \ref{sym}, $\forall\; \rho <8\pi (1+\alpha_1)$ $\exists\; \delta_\rho,C_\rho>0$ such that 
$$
J_\rho^h(u)\ge \delta\int_{S^2}|\nabla u|^2 dv_{g_0} - C_\rho 
$$
$\forall\; u\in H_{rad, p_1}$. Thus  $J_\rho^h$ is coercive on the space
$$
\left\{u\in H_{rad,p_1},\int_{\Sigma} u\; dv_{g_0} =0\right\},
$$
and by direct methods we can find a minimum point of $J_{\rho}^h$ in $H^{1}_{rad,p}$. Since $h \in H^1_{rad,p_1}$, by Palais' criticality principle (see Remark 11.4 in \cite{AmbMal}), this minimum point is a solution of \eqref{regular}.
\end{proof}
 
As a consequence of Theorems \ref{sing pos} and \ref{teo 2} we obtain a multiplicity result for equation \eqref{regular}. Indeed we can observe that if $\rho <8\pi$ is sufficiently close to $8\pi$, one has 
$$
\min_{u\in H^1(S^2)}  J_{\rho }^h< \min_{u\in H_{rad,p_1} } J_\rho^h. 
$$
 
\begin{cor}
Suppose $h$ satisfies the hypotheses of Theorem \ref{teo 2}. There exists $\eps_0>0$ such that $\forall\; \rho \in (8\pi-\eps_0, 8\pi)$, equation \eqref{regular} has at least two solutions $u$, $v$ such that $u \in H_{rad,p_1}$ and $v\in H^1(S^2)\bs H_{rad,p_1}$. 
\end{cor}
\begin{proof}
For any $\rho<8\pi$ let us take two functions $u_\rho\in  H^1(S^2) , v_\rho \in H_{rad,p_1}$, such that
$$
J_\rho^h(u_\rho)= \min_{H^1(S^2)} J_\rho^h, \qquad J^h_\rho(v_\rho)= \min_{H_{rad,p_1}(S^2)} J^h_\rho(u)\quad \mbox{ and } \quad \int_{\Sigma} u_\rho dv_{g_0} = \int_{\Sigma} v_\rho dv_{g_0} =0.
$$ 
We claim that, for $\eps$ sufficiently small and  $\rho \in (8\pi-\eps, 8\pi)$, $u_\rho \notin H_{rad,p_1}$ and in particular $u_\rho \neq v_\rho$. Assume by contradiction that there exists a sequence $\rho_n \nearrow 8\pi$ for which $u_{\rho_n}  \in H_{rad,p_1}$. Then, applying Lemma \ref{sym} as in the proof Theorem \ref{teo 2}, we would have 
$$
J^h_{\rho_m}(u_{\rho_m}) \ge \delta\int_{S^2} |\nabla u_{\rho_n}|^2 dv_{g_0} -C
$$
for some $\delta, C>0$. Therefore $\|\nabla u_{\rho_n}\|_{2}$ would be uniformly bounded and, up to subsequences, $u_{\rho_n}\rw u$ in $H^1(S^2)$ with $J_{8\pi}^h(u)= \inf_{H^1(S^2)} J_{8\pi}^h $. This is not possible because we know by Theorem \ref{sing pos} that $J_{8\pi}^h$ has no minimum point. 
\end{proof}

Now we will discuss some sufficient conditions for  the existence of solutions of \eqref{regular}, without symmetry assumptions on $h$. Let $H_0$, $T_\rho$, $d_\rho$ and $\Gamma(\alpha_1,\ldots,\alpha_m)$ be defined as in \eqref{spazioH0}, \eqref{Trho}, \eqref{drho} and \eqref{sing set}. First of all we recall a well known result concerning blow-up analysis for sequences of solutions.

\begin{prop}[See \cite{bar-tar}, \cite{bar-mont}]\label{BarTar}
Let $(\Sigma,g)$ be a compact Riemannian surface and let  $h$ be a function satisfying \eqref{h e K} with $K\in C^{\infty}_+(\Sigma)$. If $u_n$ is a sequence of solutions of \eqref{regular} on $\Sigma$ with $\rho= \rho_n \ra \ov{\rho}$ and $\int_{\Sigma} u_n dv_g =0$, Then, up to subsequences, one of the following holds:
\begin{itemize} 
\item[(i)] $|u_n| \le C$ with $C$ depending only on $\alpha_1,\ldots,\alpha_m$, $\max_{\Sigma} K$, $\min_{\Sigma} K$ and $\ov{\rho}$. 
\item[(ii)]\label{caso2} (blow-up). There exists a finite set $S=\{q_1,\ldots,q_k\}$  such that $u_n\ra -\infty$ uniformly on compact subsets of $\Sigma\bs S$. Moreover $\frac{h u_n}{\int_{\Sigma} h e^{u_n} dv_{g}} \rw \sum_{i=1}^k \beta_i \delta_{q_i} $ with $\beta_i = 8\pi$ if $q_i \in \Sigma\bs \{p_1,\ldots,p_m\}$ and $\beta_i = 8\pi(1+\alpha_j)$ if $q_i = p_j$ for some $1\le j\le m$.  
\end{itemize}
\end{prop}

Clearly case \emph{(ii)} is possible only if $\ov{\rho}\in \Gamma(\alpha_1,\ldots,\alpha_m)$.  As a direct consequence of Proposition \ref{BarTar} we get that, if $E$ is a compact subset of $(0,+\infty)\bs \Gamma(\alpha_1,\ldots,\alpha_n)$,  the set of all the solutions of \eqref{regular} in $H_0$  with $\rho \in E$ is a bounded subset $H_0$. This bound depends only on $E$, $\alpha_1,\ldots,\alpha_m$ and on $\max_{\Sigma} K,\min_{\Sigma} K$, thus, using the homotopy invariance of the Leray-Schauder degree, one can prove that, if $R$ is chosen sufficiently large, $d_\rho$ is well defined and does not depend on $R$ and $K$. Moreover $d_\rho$ is constant on every connected component of  $(0,+\infty)\bs \Gamma(\alpha_1,\ldots,\alpha_n)$. In \cite{ChenLin4} Chen and Lin introduced the generating function 
$$
g(x):=  (1+x+x^2+x^3\ldots)^{m-2}\prod_{i=1}^m(1-x^{1+\alpha_i}) 
$$
and observed that 
\begin{equation}\label{coef}
g(x)= 1+\sum_{j=1}^\infty b_j x^{n_j}
\end{equation}
where  $n_1<n_2<n_3<\ldots$ are such that 
$$
\Gamma(\alpha_1,\ldots,\alpha_m) =\{8\pi n_j\;:\; j \ge 1\}. 
$$

\begin{Te}[\cite{ChenLin4}]\label{grado}
Let $h$ be a function satisfying \eqref{h e K}, then for $\rho \in (8\pi n_k,8\pi n_{k+1})$ we have
$$
d_\rho = \sum_{j=0}^k b_j
$$
where $b_0=1$ and $b_j$ are the coefficients in  \eqref{coef}.
\end{Te}

As a consequence of this formula, \eqref{regular} has a solution for any $\rho \in (0,8\pi(1+\alpha_1))\bs 8\pi \N$.

\begin{lemma}\label{BarMal}
Suppose that $h$ satisfies \eqref{h e K} with $K\in C^\infty_+(S^2)$, $m\ge 2$ and $0<\alpha_1\le \ldots\le \alpha_m$. Then equation \eqref{regular} has a solution $\forall\; \rho \in (0,8\pi(1+\alpha_1))\bs 8\pi \N$.
\end{lemma} 
\begin{proof}
Indeed the first negative coefficient appearing in the expansion 
$$g(x)= (1+x+x^2+x^3\ldots)^{m-2}\prod_{i=1}^m(1-x^{1+\alpha_i}) = 1 + \sum_{j=1}^\infty b_j x^{n_j}$$
is the coefficient of $x^{1+\alpha_1}$, i.e.
$$
g(x)= \sum_{j=0}^\infty b_j x^{n_j}
$$
with $b_0=1$ and $b_j\ge 0$ for any $j\ge 1$ such that $n_j<1+\alpha_1$. 
 From Theorem \ref{grado} it follows that  $d_\rho\ge 1 $ for $\rho \in (0,8\pi(1+\alpha_1))\bs 8\pi \N$.
\end{proof}

\begin{oss}
Lemma \ref{BarMal} only holds for $m\ge 2$. Indeed for $m=1$ and $K\equiv 1$ one can use a Pohozaev-type identity (see \cite{mio2}) to prove that \eqref{regular} has no solutions for $\rho \in [8\pi,8\pi(1+\alpha_1)]$.
\end{oss}

\begin{oss}
A different proof of Lemma \ref{BarMal} was given in \cite{BarMal} by Bartolucci and Malchiodi using topological methods. 
\end{oss}

By Proposition \ref{BarTar}, if $\rho_n\ra 8k \pi $  with $k<1+\alpha_1$, then any blowing-up sequence of solutions of \eqref{regular} must concentrate around exactly $k$ points $q_1,\ldots,q_k\in \Sigma \bs\{p_1,\ldots,p_m\}$. A more precise description of the blow-up set is given in \cite{ChenLin1} (see also \cite{ChenLin3}, \cite{ChenLin4}):

\begin{prop}[\cite{ChenLin1}, \cite{ChenLin3}]\label{Stime CL}
Let $u_n$ be a sequence of solutions of \eqref{regular} with  $\rho=\rho_n \ra 8\pi k$ and $k<1+\alpha_1$. If alternative \emph{(ii)} of Proposition \ref{BarTar} holds, then $u_n$ has exactly $k$ blow-up points $q_1,\ldots,q_k\in \Sigma\bs \{p_1,\ldots,p_m\}$ and $(q_1,\ldots,q_k)$ is a critical point of the function 
$$
f_h(x_1,\ldots,x_k) := \sum_{j=1}^k \left( \log h(x_j) +\sum_{l\neq j} G(x_l,x_j) \right)
$$
in $$
\{(x_1,\ldots,x_k)\in (S^2)^k\;:\; x_i \neq x_j \; \mbox{ for } i\neq j\}.
$$
Moreover we have
$$
\rho_n -8k \pi =  \sum_{j=1}^k h(q_{j,n})^{-1} \left( \Delta_{g_0} \log h (q_{j,n}) +2 (k-1)\right) \frac{ \lambda_{j,n}}{e^{\lambda_j,n}} +O(e^{-\lambda_{j,n}})
$$
where  $q_{j,n}$  are the local maxima of $u_n$ near $q_j$ and $\lambda_{j,n}=u_n(q_{j,n})$. 
\end{prop}

\begin{proof}[Proof of Theorems \ref{teo 3} and \ref{teo 31}]
Take a sequence $\rho_n \searrow 8k\pi$ and a solution $u_n\in H_0$ of \eqref{regular} for $\rho= \rho_n$. By   Propositions \ref{BarTar}, \ref{Stime CL} and standard elliptic estimates, either $u_n$ is uniformly bounded in $W^{2,q}(S^2)$ for any $q\ge 1$ or  $u_n$ blows-up at $(q_1,\ldots,q_k)\in \Sigma \bs\{p_1,\ldots,p_m\}$.   In the former case we have $u_n\ra u$ in $H^1(S^2)$ and $u$ satisfies \eqref{regular} with $\rho=8\pi k$. The latter case can be excluded using \eqref{cond 1}, \eqref{cond 2}. Indeed we have
$$
\Delta_{g_0}\log  h(q_j)+2(k-1) = \Delta_{g_0} \log K -\sum_{i=1}^m\alpha_i+2(k-1)<0
$$
for any $j$. Denoting $q_{j,n}$ the maximum point of $u_n$ near $q_j$ and $\lambda_{j,n}=u_n(q_{j,n})$, by Proposition \ref{Stime CL} we get
$$
\rho_n -8\pi k =  \sum_{j=1}^k h(q_{j,n})^{-1} \left( \Delta_{g_0} \log h (q_{j,n}) +2 (k-1)\right) \frac{ \lambda_{j,n}}{e^{\lambda_j,n}} +O(e^{-\lambda_{j,n}})=$$
$$
=\sum_{j=1}^k h(q_{j})^{-1} \left( \Delta_{g_0} \log h (q_{j}) +2 (k-1)\right) \lambda_{j,n} e^{-\lambda_j,n} +o(\lambda_{j,n} e^{-\lambda_{j,n}}) <0
$$
which contradicts $\rho_n \searrow 8k \pi $.
\end{proof}

In order to prove Theorems \ref{teo 4}, \ref{teo 5} we need to compute the Leray-Schauder degree for $\rho =8\pi$.

\begin{lemma}
Let $h$ be a function satisfying \eqref{h e K} with $K\in C^\infty_+(\Sigma)$  and $\alpha_1,\ldots,\alpha_m>0$.  If $\Delta_{g_0} h(q)\neq 0$ for any $q\in \Sigma\bs\{p_1,\ldots,p_m\}$ critical point of $h$, then $d_{8\pi}$ is well defined.
\end{lemma}
\begin{proof}
It is sufficient to prove that the set of solutions of \eqref{regular} in $H_0$ with $\rho =8\pi$ is a bounded subset of $H_0$. Assume by contradiction that there exists $u_n \in H_0$ solution of \eqref{regular} for $\rho= 8\pi$ such that $\|u_n\|_{H_0}\ra +\infty$. By Propositions \ref{BarTar} and \ref{Stime CL}, there exists $q\in \Sigma \bs \{p_1,\ldots,p_m\}$ such that $u_n \rw 8\pi \delta_q$, $\nabla h (q)=0$ and 
$$
0= h(q_n)^{-1}\Delta_{g_0} \log h(q_n) \lambda_n e^{-\lambda_n}  +O(e^{-\lambda_n}) = h(q)^{-2} \Delta_{g_0} h(q) \lambda_n e^{-\lambda_n} +o(\lambda_n e^{-\lambda_n})
$$
where $\lambda_n:= \max_{\Sigma}u_n$ and $u_n(q_n)=\lambda_n$. Since $\Delta_{g_0} h(q)\neq 0$ this is not possible. 
\end{proof}

Under nondegeneracy assumptions, Chen and Lin proved that for any critical $q$ point of $h$ there exists a  blowing-up sequence of solutions which concentrates at $q$. Moreover they were able to compute the total contribution to the Leray-Schauder degree of all the solutions concentrating at $q$.

\begin{prop}[see \cite{ChenLin2}, \cite{ChenLin4}]\label{rozza}
Assume that $h$ is a Morse function on $\Sigma\bs\{p_1,\ldots,p_m\}$. Given a critical point $q\in \Sigma \bs\{p_1,\ldots,p_m\}$ of $h$, the total contribution to $d_{8\pi}$ of all the solutions of \eqref{regular} concentrating at $q$ is equal to $\sgn(\rho-8\pi)(-1)^{\ind_p}$, where $\ind_{p}$ is the Morse index of $p$ as critical point of $h$.  
\end{prop}

\begin{proof}[Proof of Theorems \ref{teo 4}, \ref{teo 5}]
Let us denote 
$$
\Lambda_- =\left\{q \in \Sigma\bs\{p_1,\ldots,p_m\}\;:\; \nabla h (q) =0,\;\Delta_{g_0} h (q)<0\right\},
$$
$$
\Lambda_+ =\left\{q \in \Sigma\bs\{p_1,\ldots,p_m\}\;:\; \nabla h (q) =0,\;\Delta_{g_0} h (q)>0\right\}.
$$
By Proposition \ref{rozza} we have 
$$
d_{8\pi} = 1 - \sum_{q\in \Lambda_-}(-1)^{\ind_q} = \ov{d}+ \sum_{q\in \Lambda_+}(-1)^{\ind_q},
$$
where $\ov{d}$ is the Leray-Schauder degree for $\rho \in (8\pi, 8\pi +\eps)$. Clearly $\Lambda_-$ contains only the local maxima of $h$ and the saddle points of $h$ in which $\Delta_{g_0} h <0$, thus
$$
d_{8\pi} = 1- r+s.
$$
Therefore we get existence of solutions if $r\neq s+1$. Similarly we have 
$$
d_{8\pi} = \ov{d} -s' + r'
$$
and we get solutions if $s' \neq r' +\ov{d}$. $\ov{d}$ can be computed using Theorem \ref{grado}. If $m\ge2$,
$$
g(x)=1+x+\cdot... \quad \Longrightarrow \quad \ov{d}=2.
$$
If $m=1$ we have 
$$
g(x):=1-x -x^{1+\alpha} + x^{2(1+\alpha)} \quad\Longrightarrow \quad \ov{d}=0.
$$
If $m=0$, then 
$$
g(x)= 1-2x +x^2 \quad\Longrightarrow \quad \ov{d}=-1.
$$
This concludes the proof.
\end{proof}

\bibliographystyle{plain}
\bibliography{bib}
\end{document}